\documentclass[11pt]{amsart}

\usepackage[PostScript=dvips]{diagrams}

\usepackage{tikz-cd}
\usepackage{amsmath,amstext,amscd, amssymb}
\usepackage{enumerate}
\usepackage{epsfig, psfrag, color, multicol, graphicx}
\usepackage{hyperref} 
\hypersetup{colorlinks} 

\newcommand{\ra}{\rangle}
\def\<{\langle}
\def\>{\rangle}
\def\cG{\mathcal G}
\def\.{\hskip.06cm}
\def\ts{\hskip.03cm}

\def\ssu{\subset}

\def\wt{\widetilde}
\def\rT{{\mathbf{T}}}
\def\zz{\mathbb Z}
\def\zzz{{{\zz_2}}}
\def\nn{\mathbb N}

\def\rr{\mathbb R}
\def\qqq{\mathbb Q}

\DeclareMathOperator{\PSL}{PSL}
\def\ii{{\text{\textbf{\em i}}}}

\newcommand{\toto}{{\ts\twoheadrightarrow\ts}}
\DeclareMathOperator{\bp}{{{{ \bigotimes \ts }}}}

\def\op{{{{ \ts \otimes \ts }}}}

\def\Ga{\Gamma}

\def\Si{\Sigma}
\def\ga{\gamma}
\def\si{\sigma}

\def\al{\alpha}
\def\be{\beta}
\def\om{\omega}
\def\k{\kappa}

\def\ve{\varepsilon}

\def\ff{{\rm \textbf{F}}}
\newcommand{\F}{\ff}

\theoremstyle{plain}

\newtheorem{thm}{Theorem}[section]
\newtheorem{lemma}[thm]{Lemma}

\newtheorem{cor}[thm]{Corollary}
\newtheorem{prop}[thm]{Proposition}

\newtheorem{conj}[thm]{Conjecture}
\newtheorem{question}[thm]{Question}

\newtheorem{clm}[thm]{Claim}
\theoremstyle{remark}
\newtheorem{ex}[thm]{Example}

\newtheorem{rem}[thm]{Remark}

\theoremstyle{definition}
\newtheorem{defn}[thm]{Definition}

\newtheorem{Open questions}[thm]{Open questions}
\newtheorem{Open question}[thm]{Open question}
\newtheorem{Open problems}[thm]{Open problems}
\newtheorem{Open problem}[thm]{Open problem}
\newtheorem{OP}[thm]{Open Problem}

\newcommand{\bG}{\mathbf{G}}

\DeclareMathOperator{\Aut}{Aut}

\usepackage[colorinlistoftodos]{todonotes}

\RequirePackage{cleveref}
\usepackage{hypcap}
\hypersetup{colorlinks=true, citecolor=darkblue, linkcolor=darkblue}
\definecolor{darkblue}{rgb}{0.0,0,0.7}

\definecolor{darkred}{rgb}{0.68,0,0}

\definecolor{darkgreen}{rgb}{0,.38,0}

\def\emp{\nothing}

\def\zz{\mathbb Z}
\def\nn{\mathbb N}

\def\rr{\mathbb R}
\def\qqq{\mathbb Q}

\def\ff{\mathbb F}

\def\pp{\mathbb P}

\def\ov{\overline}

\def\sm{\smallsetminus}

\def\Ga{\Gamma}

\def\Si{\Sigma}

\def\ga{\gamma}
\def\si{\sigma}

\def\al{\alpha}
\def\be{\beta}
\def\om{\omega}
\def\k{\kappa}

\def\ve{\varepsilon}

\def\cH{\mathcal H}

\def\ssu{\subset}

\def\wt{\widetilde}
\def\<{\langle}
\def\>{\rangle}

\def\rH{ {\text {\rm H} } }

\def\rT{{\text {\rm T} } }

\def\site{{\text {\rm s} } }
\def\bond{{\text {\rm b} } }

\def\0{{\mathbf 0}}

\def\nothing{\varnothing}

\def\.{\hskip.06cm}
\def\ts{\hskip.03cm}

\newcommand{\supp}{\mathrm{supp}}

\def\.{\hskip.06cm}
\def\ts{\hskip.03cm}

\title[Monotone parameters on Cayley graphs]{Monotone parameters on Cayley graphs \\ of finitely generated groups}
\author[Martin~Kassabov]{ \ Martin~Kassabov$^\star$}

\author[Igor~Pak]{ \ Igor~Pak$^\diamond$}

\thanks{\thinspace ${\hspace{-.45ex}}^\star$Department of Mathematics, Cornell University, Ithaca, NY 14853, USA;
\texttt{martin.kassabov@gmail.com}}

\thanks{\thinspace ${\hspace{-.45ex}}^\diamond$Department of Mathematics,
UCLA, Los Angeles, CA 90095, USA;
\texttt{pak@math.ucla.edu}}

\date \today

\begin{document}

\begin{abstract}
We construct a new large family of finitely generated groups with
continuum many values of the following monotone parameters:
spectral radius, critical probabilities, and asymptotic entropy.
We also present several open problems on other monotone parameters.
\end{abstract}

\maketitle

\vskip.6cm

\section{Introduction}\label{s:intro}

\smallskip

\subsection{Main results}\label{ss:intro-main}
There are several probabilistic parameters of Cayley graphs
of finitely generated groups
that capture ``global properties'' of the groups, i.e.,\ independent
of the generating sets (see below and Section~\ref{s:other}).
Typically, these parameters are \emph{monotone}:
as groups get ``larger'' the parameters increase/decrease, usually in
a difficult to control way.

While there is an extensive literature on bounds for these parameters
and their relations to each other, computing them remains challenging
and the exact values are known in only few examples and special families.
Our main result shows that these parameters can be as unwieldy as the finitely
generated group themselves.

\smallskip

\begin{thm}[Main theorem] \label{t:main}
Let $G=\<S\>$ be a finitely generated group, and let $f(G,S)$ denote
one of the following parameters:

\smallskip

$\circ$ \ spectral radius \ts $\rho(G,S)$,

$\circ$ \ asymptotic entropy \ts $h(G,S)$,

$\circ$ \ site percolation critical probability \ts $p^\site_c(G,S)$,

$\circ$ \  bond percolation critical probability \ts $p^\bond_c(G,S)$.

\smallskip

\noindent
Then, there is a family of \ts $4$-regular Cayley graphs
\. $\big\{\ts \text{\em Cay}(G_\om,S_\om)\ts\big\}$, such that the set of  parameter values \.
$\big\{f(G_\om,S_\om)\big\}$ \. has cardinality of the continuum.
\end{thm}

The construction of this family of Cayley graphs is based on
properties of an uncountable family of \emph{decorated Grigorchuk groups},
the setting we previously considered in~\cite{KP} and related to the
approach in~\cite{TZ19}, see~$\S$\ref{ss:finrem-hist}.
See also $\S$\ref{ss:finrem-small} for an alternative approach
to the theorem.  We also obtain the following result of independent interest.

\smallskip

\begin{thm}\label{t:isol}
In notation of Theorem~\ref{t:main}, for \ts $k\ge 3$ \ts the sets \.
$$
X_{\rho,k} \. := \. \big\{\rho(G,S) \. : \. |S|=k \big\}
$$
have no isolated points on \ts $[\al_k,1]$, \ts where \.
$\al_k :=\frac{2\sqrt{k-1}}{k}$\..
\end{thm}

\smallskip

Note that \ts $X_{\rho,k} \subseteq [\al_k,1]$, so the theorem implies that \ts $X_{\rho,k}$ \ts
has no isolated points, except possibly at~$\al_k$\ts.
Here \. $\al_k$ \. is the spectral radius of the infinite $k$-regular tree, which can be
viewed as the Cayley graph of a free product of \ts $k$ \ts copies of \ts $\zz_2$\ts.
A version of Theorem~\ref{t:isol} for the asymptotic entropy was given
by Tamuz and Zheng in \cite{TZ19}.

\smallskip

\subsection{Monotone parameters} \label{ss:intro-mono}
Let \ts $f: \{(G,S)\} \to \rr$ \ts be a function on \emph{marked groups},
i.e.,\ pairs of finitely generated groups~$G$ and finite symmetric
generating sets \ts $S=S^{-1}$.
We refer to $f$ as \emph{parameters} (see the discussion in~$\S$\ref{ss:finrem-par}).
%
We say that $f$ is \emph{decreasing} \ts if the following conditions hold:
\begin{equation}\label{eq:mono}\tag{$\ast$}
f(G/N,S') \, \ge \, f(G,S) \qquad \text{for all \ \ $N \lhd G$.}
\end{equation}
\begin{equation}\label{eq:mono-limit}\tag{$\ast\ast$}
(G_n,S_n) \. \underset{n\to \infty}{\to} \. (G,S) \ \. \Longrightarrow \. \ \limsup_{n\to \infty} \. f(G_n,S_n) \. \ge \. f(G,S).
\end{equation}
Here \ts $N \lhd G$ \ts denotes a normal subgroup~$N$ of~$G$,
and $S'$ is a projection of~$S$ onto~$G/N$.  In~\eqref{eq:mono-limit}, the
convergence on the left\footnote{The
inequality~\eqref{eq:mono-limit} states that \ts $f$ \ts is \emph{semi-continuous},
see Lemma~\ref{lm:semicontunuous}; we include it in the definition
to emphasize the direction of the inequality.}
is in the Chabauty topology for marked groups,
see~$\S$\ref{ss:marked-limits}.

We say that $f$ is \emph{increasing} if the inequalities \eqref{eq:mono} and~\eqref{eq:mono-limit}
are reverted.  We say that $f$ is \emph{monotone} if it is either increasing or decreasing.
We say that $f$ is \emph{strictly decreasing} (respectively, \emph{strictly increasing} and
\emph{strictly monotone}) if the inequality~\eqref{eq:mono} is strict,
provided that $N$ is nonamenable.  Similarly, we say that $f$ is \emph{sharply decreasing}
(respectively, \emph{sharply increasing} and
\emph{sharply monotone}) if the inequality~\eqref{eq:mono} is strict \ts
{if and only if} \. $N$ \ts is nonamenable (see also~$\S$\ref{ss:finrem-par}).

\subsection{Spectral radius}\label{ss:intro-radius}
Let \ts $\text{Cay}(G,S)$ \ts denote the Cayley graph of a finitely generated
group $G=\<S\>$.    The \emph{spectral radius} \ts $\rho(G,S)$ \ts
is defined as:
$$
\rho(G,S) \,  := \, \limsup_{n\to \infty} \. \frac{1}{|S|} \sqrt[n]{c(n)} \.,
$$
where the \emph{cogrowth sequence} \ts $\big\{c(n)=c(G,S; \ts n)\big\}$ \ts is the number of words in the alphabet $S$ which are equal to
the identity~$e$ in~$G$.
Equivalently, this is the number of loops in \ts $\text{Cay}(G,S)$ \ts of length~$n$,
starting and ending at~$e$.

Famously, it was shown by Kesten~\cite{Kes},
that $\rho(G,S) =1$ if and only if $G$ is nonamenable, and that
$$
\al_k \. = \, \tfrac{2\. \sqrt{k-1}}{k} \, \le \,
\rho(G,S) \, \le \. 1\,, \quad \text{where} \ \ k=|S|.
$$
For \ts $k=2m$, the lower bound is 
attained on a
free group \. $\F_m=\<x_1^{\pm 1},\ldots,x_m^{\pm 1}\>$.

Unfortunately, relatively little is known about the spectral radius
in full generality, beyond these basic inequalities.  Notably,
it is open whether \emph{every} \ts $\al \in (0,1]$ \ts is a spectral
radius of \emph{some} \ts finitely generated group, cf.~\cite[Question~4.2]{BLM}.
On the other hand, until this paper it was not known if there is a single example of a
group with a transcendental spectral radius, a problem discussed
in~\cite[$\S$2.4]{Pak18}.

For several families of nonamenable groups, the exact value of $\rho(G,S)$ is known,
see e.g.~\cite{GH,Woe} (see also~\cite{BLM,ERVW,Kuk99}).
In all these cases the spectral radii are algebraic.
Let  \ts $\Ga_2$ \ts denote the surface group of genus~$2$, i.e.
$$
\Ga_2 \, = \, \<a_1,a_2,b_1,b_2 \. \mid \. a_1b_1a_1^{-1}b_1^{-1}a_2b_2a_2^{-1}b_2^{-1} = 1\>.
$$
After much effort, it was shown that
\. $0.6624 \le \rho(\Ga_2,S) \le 0.6629$, where \ts $S$ \ts
are standard generators as above (see \cite[$\S$7]{GH} and references
therein).  Sarnak's question whether this spectral radius is transcendental
remains unresolved (ibid.)

\begin{prop}\label{p:radius-mono}
Spectral radius \. $\rho: \{(G,S)\} \to (0,1]$ \. is a decreasing parameter.
\end{prop}

Indeed, the property~(\ref{eq:mono}) is straightforward, while the property~(\ref{eq:mono-limit})
uses the fact that \ts $\rho(G,S)$ \ts is a limit of a supermultiplicative sequence: \ts $c(n+m) \geq c(n)c(m)$.
 The following result by Kesten
shows that the spectral radius is sharply decreasing:

\begin{lemma}[{\rm \cite[Lemma~3.1]{Kes}}]
\label{l:radius-strict}
Let \ts $G=\<S\>$ \ts and let $N$ be a 
normal subgroup of~$G$ and $S'$ be the projection of $S$ onto~$G/N$.
Then \ts $\rho(G,S) < \rho(G/N,S')$ \. \underline{if and only if} \.
$N$ \ts is nonamenable.
\end{lemma}

Note that the corresponding inequality for the Cheeger
constant remains open (Conjecture~\ref{conj:isop-strict}).
Note also that there is an alternative notion of the \emph{cogrowth
sequence}, where only reduced words are considered (equivalently,
paths on the Cayley graph are not allowed to backtrack),
see e.g.\ \cite[$\S$14.6]{CD21}.  Much of what we present
translates easily to this setting; we omit it to
avoid the confusion.

\subsection{Critical probabilities} \label{ss:intro-pc}
In the \emph{Bernoulli site percolation}, the vertices of the graph
\ts $\text{Cay}(G,S)$ \ts are open with probability $p$ and closed
with probability $(1-p)$, independently at random.  The \emph{Bernoulli bond percolation}
is defined analogously, but now the edges are open/closed. 
These notions are different, but closely related to each other, see below.

Denote by \ts $\theta^\site(G,S,p)$ \ts and \ts $\theta^\bond(G,S,p)$ \ts
the probability that the identity element~$e$ is in an infinite connected component in the
site and bond percolation, respectively.  We omit the superscript when
the notation or results hold for both site and bond percolations,
hoping this would not lead to confusion.

The \emph{critical probability} is defined as follows:
$$
p_c(G,S) \. := \. \sup \big\{ \ts p \, : \, \theta(G,S,p)=0\ts\big\}.
$$
We refer to~\cite{BS} for an introduction to percolation on Cayley graphs,
to~\cite{BR06,Gri99,Wer09} for a thorough treatment of both classical and recent
aspects, and to~\cite{Dum18} for a recent overview of the subject
(see also~$\S$\ref{ss:finrem-perc}).

It is easy to see that \ts $p_c=1$ \ts for all groups of linear
growth (which are all virtually~$\zz$).  It was conjectured in \cite[Conj.~2]{BS},
that \ts $p_c<1$ for all groups of superlinear growth.  Special cases of
this conjecture have been established in a long series of papers, until it
was eventually proved in~\cite{D+}. 
The ultimate result, the remarkable ``gap inequality'' \. $p_c\le 1-\ve$ \. for a
universal constant \ts $\ve>0$, was obtained in~\cite{PS23} for all groups of
superlinear growth, see also \cite{HT21}.   This shows that Theorem~\ref{t:isol}
does not apply to critical probabilities.

Famously, it was shown in \cite{GS98}, that
$$p_c^\bond(G,S) \. \le \. p^\site_c(G,S) \. \le \. 1 - \big(1-p_c^\bond\big)^{k-1} \quad \text{where} \quad k\ts = \ts |S|.
$$
It is also known that bond percolation can be simulated by site percolation,
but not vice versa, see \cite{GZ24}.  In general, these critical probabilities
do not coincide.  For example, the celebrated Kesten's theorem
states that \ts $p_c^\bond = \ts \frac{1}{2}$ \ts for the square grid (Cayley graph
of~$\zz^2$ with standard generators).
By comparison, $p^\site_c \approx 0.592746$ in this case~\cite{NZ01},
although the exact value is not known.

In all known examples when the critical probabilities \ts $p^\site_c(G,S)$ \ts and \ts $p^\bond_c(G,S)$ \ts are
computed exactly, they are always algebraic, cf.~\cite{Koz,SZ10}.  For example, it is known that
$$p^\site_c(G,S) \, = \, p_c^\bond(H,R) \, = \, 1 \. - \. 2 \ts \sin \big(\tfrac{\pi}{18}\big) \, \in \. \ov{\qqq}\,,
$$
where
$$\aligned
G \. &= \. \<x,y \.  \mid  \. x^3=y^3=(xy)^3=1\>, \quad S \. = \. \{x,x^{-1},y,y^{-1}\}, \\
H \. & = \. \<a,b,c  \.  \mid  \. a^2=b^2=c^2=(abc)^2=1\>, \quad R \. = \. \{a,b,c\}.
\endaligned
$$
In this case, \ts $\text{Cay}(G,S)$ \ts is the \emph{Kagom\'e lattice} \ts and
\ts $\text{Cay}(H,R)$ \ts is the \emph{hexagonal lattice},
see e.g.~\cite[$\S$5.5]{BR06}.
It was asked in~\cite{PeS}, if there exist critical probabilities
\ts $p_c(G,S)$ \ts that are transcendental.
Theorem~\ref{t:main} gives a positive answer to this question.
The theorem also resolves a closely related H\"aggstr\"om's question,
see~$\S$\ref{ss:finrem-comp}.  As before, we start with the
following observation.

\begin{prop}\label{p:perc-mono}
Critical probabilities \, $p_c^\site\., \ts p_c^\bond\.: \. \{(G,S)\} \to (0,1]$ \.
are decreasing parameters.
\end{prop}

The proof of this result is straightforward.  The following result by
Martineau and Severo shows that critical probabilities are strictly
decreasing:

\begin{lemma}[{\rm \cite{MS19}}]
\label{l:perc-strict}
Let \ts $G=\<S\>$ \ts and let \ts $N\ne \mathbf{1}$ \ts be a 
normal subgroup of~$G$.
Let $S'$ be the projection of $S$ onto~$G/N$.
Then:
$$
p_c^\site(G,S) \. < \. p_c^\site(G/N,S') \quad \. \text{and} \.
\quad p_c^\bond(G,S)  \. < \. p_c^\bond(G/N,S').
$$
\end{lemma}

Note that there is no assumption that $N$ is nonamenable,
and, in fact, the main result in \cite{MS19} is stated
in the greater generality
of quasi-transitive group actions on graphs.
This result resolved a well-known open problem by
Benjamini and Schramm \cite[Question~1]{BS}.  Thus, for example,
we have \. $1> p_c (\zz^2,S') > p_c(\zz^3,S)$, so the critical
probabilities are not {sharply} decreasing.

\subsection{Entropy} \label{ss:intro-entropy}
Let $G$ be a finitely generated group, and let \ts $\mu: G\to \rr_{\ge 0}$ \ts
be a probability distribution.  \emph{Shannon's entropy} \ts is defined as
$$
\rH(\mu) \, := \, - \sum_{g \ts \in \ts \supp(\mu)} \. \mu(g)  \. \log \mu(g).
$$
As before, let \. $G=\<S\>$, where \ts $S=S^{-1}$. 
Denote by \. $\mu_n(g) :=\pp[x_n=g]$ \. the distribution of the simple random
walk \ts $\{x_n\}$ \ts on \ts Cay$(G,S)$ \ts starting at identity \ts $x_0=e$. Finally,
let
$$
h(G,S) \. := \. \lim_{n\to \infty} \. \frac{\rH(\mu_n)}{n}
$$
denote the \emph{asymptotic entropy}, see \cite{KV83}.  Recall that \ts $h(G,S)>0$ \ts
if and only if \ts Cay$(G,S)$ \ts has the \emph{non-Liouville property} (existence
of non-constant bounded harmonic functions).  Equivalently, \ts $h(G,S)=0$ \ts
if and only if the random walk \ts $\{x_n\}$ \ts has trivial Poisson boundary, ibid.

We note that there are solvable groups of exponential growth with positive
asymptotic entropy; the \emph{lamplighter group} \. $\zz_2 \wr \zz^d$ \ts for $d \geq 3$ is
the most famous example \cite[$\S$6]{KV83} (see also \cite[$\S$9.1]{Pete}
and an introduction in~\cite{Tab17}).
Note also that the asymptotic entropy is known explicitly only in a few cases as it is
so hard to compute.  For example, it was computed in \cite[p.~22]{Gri-thesis} and
\cite[Prop.~2.11]{Bis92}, that
$$
h(\F_m,S) \, = \, \frac{m-1}{m} \. \log \ts (2m-1).
$$
Here \. $\F_m = \big\<z_1^{\pm 1},\ldots,z_m^{\pm 1}\big\>$ \.
is a free group with the standard generating set.  Note that the asymptotic entropy
is transcendental in this case: \. $h(\F_m,S)\notin \ov{\qqq}$.

\begin{prop}\label{p:entropy-mono}
Asymptotic entropy \. $h: \{(G,S)\} \to \rr_{\ge 0}$ \. is an increasing parameter.
\end{prop}

The proof of this result is straightforward.  The following result by Kaimanovich
shows that the asymptotic entropy is strictly increasing:

\begin{lemma}[{\rm \cite[Thm~2]{Kai02}}]
\label{l:entropy-strict}
Let \ts $G=\<S\>$ \ts and let $N$ be a 
normal subgroup of~$G$.  Suppose $N$ is nonamenable,
and let $S'$ be the projection of $S$ onto~$G/N$.
Then \ts $h(G,S) > h(G/N,S')$.
\end{lemma}

Since there are amenable groups with positive asymptotic entropy,
we conclude that the asymptotic entropy is not sharply increasing.
Also, note that the corresponding inequality for the speed of random walk
\ts $\{x_n\}$ \ts remains open (Conjecture~\ref{conj:speed-strict}).
Let us mention other (closely related) entropy notions, such as the
\emph{Connes--St\o{}rmer entropy} (see e.g.~\cite{Bis92}),
and the \emph{tree entropy} \cite{Lyons}.  These can also
be viewed as probabilistic parameters.

\subsection{Proof outline}\label{ss:intro-proof}
The proof of Theorem~\ref{t:main} is extremely general and is
based on a group theoretic construction and a set theoretic argument.
Formally, Theorem~\ref{t:main} is an immediate consequence from
following two complementary lemmas.

\begin{lemma}[combined lemma] \label{l:combined-mono}
All parameters~$f$ in Theorem~\ref{t:main} are strictly monotone.
\end{lemma}

The lemma is a combination of
Propositions~\ref{p:radius-mono},~\ref{p:perc-mono},~\ref{p:entropy-mono}
and Lemmas~\ref{l:radius-strict},~\ref{l:perc-strict},~\ref{l:entropy-strict}.


\begin{lemma}[main lemma] \label{l:main-mono}
Let $f$ be a strictly monotone parameter.  Then, there is a family
 \. $\big\{(G_\om,S_\om)\big\}$ \. of marked groups, s.t. \.
$\big\{f(G_\om,S_\om)\big\}$ \. has cardinality of the continuum.
\end{lemma}

The proof of Lemma~\ref{l:main-mono} is in turn an easy consequence
of the following result of independent interest.

\begin{lemma}[main construction] \label{l:main-construction}
There exists a family \ts $\big\{(G_J,S_J) \ts : \ts J \in 2^{\nn}\big\}$ \ts
of \. $4$-generated marked groups, satisfying the following:
\begin{itemize}
\item {\rm (strict monotonicity)} \, For every subset \. $J\subsetneq J'$,
there is a surjection of marked groups \. $G_{J'} \. \to \. G_J$,
s.t.\ the kernel of the projection is nonamenable.
\item {\rm (continuity)} \, For every sequence \. $\{J_n \ssu J \. : \. n\in \nn\}$
\. s.t.\ \. $J_n \to J$ \. in the Tychonoff topology of \.  $2^{\nn}$,
we have \. $\big(G_{J_n},S_{J_n}\big) \. \to \. (G_J,S_J)$ \. in the Chabauty topology.
\end{itemize}
\end{lemma}

Equivalently, the continuity condition says that there exists a function \.
$\lambda: \nn \to \nn$, such that the ball of radius $n$ in the Cayley graph
of  $(G_J,S_J)$, depend only on \. $J \cap \{1,\ldots,\lambda(n)\}$.

\subsection{Paper structure} \label{ss:intro-structure}
In Sections~\ref{s:marked} and~\ref{s:gri} we recall definitions
of marked groups, Grigorchuk groups and their convergence.
Most readers familiar with the area should be able to skip
this sections.  In Section~\ref{s:construction} we give the
proof of Lemmas~\ref{l:main-construction} and~\ref{l:main-mono},
which complete the proof of the main Theorem~\ref{t:main}.
In Section~\ref{s:para}, we prove Theorem~\ref{t:isol}.
We conclude with a discussion of other monotone parameters
in Section~\ref{s:other} and final remarks in Section~\ref{s:finrem}.

\medskip

\section{Marked groups and their limits}\label{s:marked}

We recall standard definitions of marked groups.  We stay close to \cite{KP}
from which we heavily borrow the notation and some basic results.

\subsection{Basic definitions and notation}\label{ss:marked-basic}

Denote \ts $[n]=\{1,\ldots,n\}$ \ts and \ts $\nn=\{1,2,\ldots\}$. 
Let \ts $\rr_{\ge 0} = \{x\ge 0\}$, \. $\ov \qqq$ \ts the algebraic numbers,
and \ts $\qqq_p$ \ts the $p$-adic numbers.
We use \ts $(\log a)$ \ts to denote the natural logarithm.

We use both \ts $1$ \ts and \ts $e$ \ts to denote the identity in the group, and
we use \ts $\mathbf{1}$ \ts to denote the trivial group.
Let \ts  $\zz_m=\zz/m\zz$ \ts denote the group of integers modulo~$m$, and let \ts $\F_k$ \ts
denote the free group on~$k$ generators.   Throughout the paper, all generating
sets will be finite and \emph{symmetric}: \ts $S=S^{-1}$.
We use $[x,y]=x^{-1}y^{-1}xy$ to denote the commutator of elements~$x$ and~$y$.

\subsection{Definition of marked groups and their homomorphisms}\label{ss:marked-def}
All groups we will consider will have ordered finite generating
sets of the same size~$k$.  Whenever we mention a group~$G$,
we mean a pair $(G,S)$ where $S=\{s_1,\ldots,s_k\}$ is a ordered symmetric
generating set of $G$ of size $k$.  Although Cayley graph does not depends on the order of generators,
the order is crucial for our results.  We call these \emph{marked groups},
and $k$ will always denote the size of the generating set.

By a slight abuse of notation, will often drop~$S$ and refer
to a \emph{marked group~$G$}, when $S$ is either clear from
the context or not relevant.  Note that we will also have
groups that are not marked; we hope this does not lead to confusion.

Throughout the paper, the homomorphisms between marked
groups will send one generating set to the other.  Formally,
let $(G,S)$ and $(G',S')$ be marked groups, where
$S=\{s_1,\ldots,s_k\}$ and $S'=\{s_1',\ldots,s_k'\}$.  Then
$\phi: (G,S) \to (G',S')$ is a \emph{marked group homomorphism}
if $\phi(s_j) = s_j'$, and this map on generators extends to
the (usual) homomorphism between groups: $\phi: G\to G'$.

An equivalent way to think of marked groups is as
epimorphisms \. $\F_k \toto G$ \. and \. $\F_k \toto G'$. In this picture,
maps between marked groups $G$ and $G'$ correspond to commutative diagrams:
\begin{diagram}[nohug,height=7mm]
     &              &  G      \\
\F_k & \ruOnto(2,1) & \dTo   \\
     & \rdOnto(2,1) & G'     \\
\end{diagram}

This means that for every two marked groups, there is at most one homomorphism $G\to G'$ which is necessary surjective.

\subsection{Products of groups}\label{ss:marked-product}
The \emph{direct product} of groups $G$ and $H$ is denoted $G\oplus H$,
rather than more standard $G \times H$.  This notation
allows us to write infinite product as $\ts \bigoplus \ts G_i$\.,
where all but finitely many terms are trivial
and we will typically omit the index of summation.

We denote by \ts $\prod G_i$ \ts the (usually uncountable) group of
sequences of group elements, without any finiteness conditions.
Of course, when the index set is infinite, the groups \ts
$\bigoplus \ts G_i$ \ts and \ts $\prod G_i$ \ts are not
finitely generated.

Finally, let $H \wr G = G\ltimes H^\ell$ denotes the permutation
wreath product of the groups, where \ts $G \subseteq \Si_\ell$ \ts
is a permutation group of $\ell$ letters.

\subsection{Products of marked groups}\label{ss:marked-product-marked}
Fix \. $I \subseteq \nn$, and let \. $\{(G_i,S_i), i\in I\}$ \. be a
sequence of marked groups with generating sets \. $S_i=\{s_{i1},\dots,s_{ik}\}$.
Define the $(\Gamma,S) = (\bp G_i,S_i)$ to be the subgroup of $\prod G_i$ generated by
diagonally embedding the generating sets of each \ts $G_i$, i.e.,\ \.
$\bp G_i = \< s_1,\dots, s_k\>$, where \. $s_j=\{s_{ij}\} \in \prod G_i$.

Note that $\Gamma$ comes with canonical epimorphisms
$\zeta_i: \Gamma \toto G_i$. Often the generating sets will be
clear from the context and will simply use $\Ga = \bp G_i$.
When the index set contains only $2$ elements we denote the product
by $G_1\op G_2$.

The product \. $\bp G_i$ \. can be defined by universal properties and it is the
``smallest'' marked group which surjects onto each~$G_i$.
Thus, this is equivalent to the categorical product in the category of marked groups.
We refer to~$\S$\ref{ss:finrem-hist}
for some background and references.

\subsection{Limits of groups}\label{ss:marked-limits}
We say that the sequence of marked groups \. $\{(G_i,S_i) \. : \. i \in I\}$ \.
\emph{converges in the Chabauty topology},
to a group $(G,S)$ if for any $n$ there exists $m=m(n)$ such that
such that for any $i>m$ the ball of radius $n$ in $G_i$
is the same as the ball of radius $n$ in~$G$.
We write $\lim_{i\to \infty} G_i = G$.
For Cayley graphs \ts Cay$(G_i,S_i)$ \ts rooted at identity,
this is equivalent to the \emph{Benjamini--Schramm graph convergence},
see e.g.\ \cite[$\S$19.1]{Lov12}.

Equivalently, this can be stated as follows:
if
$R_i = \ker( \F_k \toto G_i)$ and $R = \ker( \F_k \toto G)$
then
$$
\lim_{i\to \infty} R_i \cap B_{\F_k}(n) \. = \. R \cap B_{\F_k}(n)\ts,
$$
which means that  for a fixed $n$ and  sufficiently large $i$ the sets \.
$R_i \cap B_{\F_k}(n)$ \. and \.  $R \cap B_{\F_k}(n)$ \. must coincide.

\begin{lemma}[{\rm \cite[Lemma~4.6]{KP}}]\label{l:product}
Let $\{G_i\}$ be a sequence of marked groups which converge to a marked group~$G$,
and define \. $\Gamma := \bp G_i$. Then there is an epimorphism \. $\pi: \Gamma \toto G$.
Moreover, the kernel of~$\pi$ is equal to the
intersection \ts $\Gamma \cap \bigoplus G_i$\ts.
\end{lemma}

Lemma~\ref{l:product} allows us to think of $G$ as the \emph{group at infinity}
for~$\Gamma$.

\begin{ex} \label{e:finite}
We note that many group properties do not survive in the limit.
For example, it is easy to construct examples of amenable groups
with a nonamenable limit.  In fact, classic Margulis's (constant degree)
expander constructions are Cayley graphs of \ts $G_p = \PSL(2,\zz_p)$ \ts
have girth \ts $\Omega(p)$ \ts and virtually free group limit~$\PSL(2,\zz)$.
See e.g.\ \cite[$\S$11]{HLW06} and \cite[$\S$7.3]{Lub95} for more on this
and further references.
\end{ex}

\medskip

\section{Grigorchuk groups}\label{s:gri}

We now recall some definitions and results on Grigorchuk groups.
Again, we stay close to notation and definition in \cite{KP},
and note that these results can be found throughout the literature.

\subsection{Free Grigorchuk group}\label{ss:gri-free}
The \emph{free Grigorchuk group} \ts $\cG$ \ts with presentation
$$
\cG = \langle a,b,c,d \ts \mid \ts a^2=b^2=c^2=d^2=bcd=1 \ra
$$
will play a central role throughout the paper, as all our groups
and also all Grigorchuk groups are homomorphic images of~$\cG$.
This group is the free product of a group of order $2$ and
elementary abelian group of order $4$, i.e.\ \ts $\cG\simeq\zz_2 \ast \zz_2^2$\..
It contains free subgroups and is nonamenable.

\subsection{Family of Grigorchuk groups}\label{ss:gri-family}
Below we present variations on standard results on the Grigorchuk
groups $\bG_\om$. 
Rather than give standard definitions as a subgroup of $\Aut (\rT_2)$,
we define $\bG$ via its properties.  We refer to~\cite{Gri-one,GP08,dlH00}
for a more traditional introduction and most results in this
subsection.

\begin{defn}
\label{d:twist}
Let $\varphi: \cG \toto \cG$ denote the automorphism of order $3$ of the group
$G$ which cyclicly permutes the generators $b$, $c$ and $d$, i.e.,
$$
\varphi(a) = a,
\quad
\varphi(b) = c,
\quad
\varphi(c) = d,
\quad
\varphi(d) = b.
$$
\end{defn}

\smallskip

Let $\pi: \cG \toto H$ be an epimorphism, i.e., suppose group $H$
comes with generating set consisting of $4$ involutions
$\{a,b,c,d\}$ which satisfy $bcd=1$.  By $F(H)$
we define the subgroup of $H \wr \zzz  =  \zzz \ltimes (H \oplus H)$
generated by the elements $A,B,C,D$ defined as
$$
A = (\xi; \ts 1,1)\ts,
\quad
B = (1; \ts a,b)\ts,
\quad
C = (1; \ts a,c)
\quad \text{and} \ \
D = (1; \ts 1, d)\ts,
$$
where $\xi^2=1$ is the generator of~$\zzz$.
It is easy to verify that $A,B,C,D$ are involutions which satisfy $BCD=1$,
which allows us to define an epimorphism $\wt F(\pi): \cG \to F(H)$.

The construction can be twisted by the powers automorphism $\varphi$
$$
\wt F_x(\pi): = \wt F(\pi \circ \varphi^{-x}) \circ \varphi^{x}.
$$
An equivalent way of defining the group $F_x(H)$ is as the subgroups
generated by
\begin{align*}
A_0 =&  (\xi; \ts 1,1)\ts,  \quad &
B_0 =& (1; \ts a,b)\ts, \quad &
C_0 =& (1; \ts a,c)\ts, \quad & 
D_0 =& (1; \ts 1, d)\ts, \\
A_1 =&  (\xi; \ts 1,1)\ts,  \quad &
B_1 =& (1; \ts a,b)\ts, \quad &
C_1 =& (1; \ts 1,c)\ts, \quad & 
D_1 =& (1; \ts a, d)\ts, \\
A_2 =&  (\xi; \ts 1,1)\ts,  \quad &
B_2 =& (1; \ts 1,b)\ts, \quad &
C_2 =& (1; \ts a,c)\ts, \quad & 
D_2 =& (1; \ts a, d)\ts.
\end{align*}

\smallskip

\noindent
Here all groups~$H$ are marked, i.e., come with an epimorphism $\cG \toto H$.
This allows us to slightly simplify the notation as above.

\begin{prop}
Each $F_x$ is a functor form the category of homomorphic images of $\cG$ to itself,
i.e., a group homomorphism $H_1 \to H_2$ which preserves the generators induces,
a group homomorphism $F_x(H_1) \to F_x(H_2)$.
\end{prop}

\begin{prop}
\label{pro:F_preserves_products}
The functors $F_x$ commutes with the products of marked groups, i.e.,
$$
F_x\left(\bp H_j\right) = \bp F_x(H_j).
$$
\end{prop}

\begin{proof}
This is immediate consequence of the functoriality of $F_i$ and the universal
property of the products of marked groups. Equivalently one can check directly
from the definitions.
\end{proof}

\begin{defn}
One can define the functor $F_\om$ for any finite word $\om \in \{0,1,2\}^*$
as follows
$$
F_{x_1x_2\dots x_i}(H):= F_{x_1}(F_{x_2}(\dots F_{x_i}(H)\dots))
$$
If $\omega$ is an infinite word on the letters $\{0,1,2\}$ by
$F_\omega^i$ we will denote the functor $F_{\omega_i}$ where
$\omega_i$ is the prefix of $\omega$ of length $i$.
\end{defn}

In~\cite{Gri3}, Grigorchuk defined a group $\bG_\omega$ for any
infinite word $\omega$.  One way to define these groups
is by $\bG_{x\omega} = F_x(\bG_\omega)$, where $x$ is any
letter in $\{0,1,2\}$.  The \emph{first Grigorchuk group}
is denoted $\bG = \bG_{(012)^\infty}$, which
corresponds to a periodic infinite word, see e.g.~\cite{Gri3,Gri-one}.

\subsection{Contraction in Grigorchuk groups} \label{ss:gri-conv}
Let $\bG_{\omega,i} = F^i_\omega(\mathbf{1})$, where $\mathbf{1}$ denotes the
trivial group with one element (with the trivial map
$\cG \toto \mathbf{1}$).

\begin{prop}[{\rm \cite[Prop.~5.9]{KP}}]
\label{p:grig-rooted}
There is a canonical epimorphism $\bG_\omega \toto \bG_{\omega,i}$.
For every $i$, the groups $\bG_{\omega,i}$ are finite and naturally act on finite binary rooted tree
of depth $i$ and this action transitive on the leaves. These actions comes from the standard action
of the Grigorchuk group on the infinite binary tree~$\rT_2$.
\end{prop}

Here the group $F^i_\om(H)$ is a subgroup of the permutational wreath product
$H \wr_{X_i} \bG_{\om,i}$, where $X_i$ is the set of leaves of the
binary tree of depth~$i$.

\begin{lemma}[{\rm \cite[Lemma~5.11]{KP}}]
\label{l:contracting_property_of_G}
Let $\pi:\cG \toto H$ be an epimorphism, i.e., group~$H$ is
generated by $4$ \emph{nontrivial} involutions which satisfy $bcd=1$.
If the word \ts $\om\in \{0,1,2\}^\ast$ \ts does not stabilize,
then the balls of radius \ts $\leq 2^{m}-1$ \. in the groups
$F^m_{\om}(H)$ and $\bG_{\om}$ coincide.
\end{lemma}

We conclude with an immediate corollary of the
Proposition~\ref{p:grig-rooted} and
Lemma~\ref{l:contracting_property_of_G},
which can also be found in~\cite{Gri5}.

\begin{cor}
\label{cor:contracting_property_of_G}
Let \. $\{\cG \toto H_i\}$ \.  be any sequence of  groups generated
by $k=4$ nontrivial involutions. Then the sequence of marked groups
converges: \. $\lim_{i\to \infty} F^{i}_\omega(H_i)\ts = \ts \bG_\omega$.
\end{cor}

\medskip

\section{Main construction} \label{s:construction}

\subsection{A nonamenable group}
\label{def_of_H}

Our main construction uses that the free Grigorchuk group $\cG$ is close to a free group and thus has many nonamenable quotients which are very different from the groups $\bG_\omega$. One such quotient is generated by following matrices in $\PSL\bigl(2,\zz[\ii,1/2]\bigr)$,  where $\ii^2=-1$,
$$
a = \left(\begin{array}{cc} \ii & \ii/4 \\ 0   & -\ii  \end{array}\right),
\quad \
b = \left(\begin{array}{cc}   0 & \ii   \\ \ii & 0   \end{array}\right),
\quad \
c = \left(\begin{array}{cc}   0 & 1     \\  -1 & 0    \end{array}\right),
\quad \
d = \left(\begin{array}{cc} \ii & 0     \\  0  & -\ii \end{array}\right).
$$
A direct computation shows that \. $a^2=b^2=c^2=d^2=bcd=1$, i.e., there is a (non-surjective)
homomorphism \.
$\iota : \cG \to \PSL\bigl(2,\zz[\ii,1/2]\bigr)$.

Let \. $\cH:=\<a,b,c,d\>$ \. denote the marked group generated by the above matrices;
as always we consider it as marked group. The group $\cG$ contains a normal subgroup of index (at most) $2$ generated by $\{c, ad\}$,%
\footnote{The group $\cG$ is $4$ generated and the standard Schreier algorithm gives that any subgroup of index $2$ is generated by $7$ elements - in these case several of these generators are trivial, and other are redundant.}
which yields a subgroup $N:=\<c,ad\>$ of index $2$ of $\cH$. The generators of this subgroup are
$$
c = \left(\begin{array}{cc}   0 & 1     \\  -1 & 0    \end{array}\right),
\quad \
ad = \left(\begin{array}{cc} 1 & -1/4 \\ 0   & 1  \end{array}\right).
$$
It can be verified that \. $\<c,ad\>=\PSL\bigl(2,\zz[1/2]\bigr)$, and thus we have \. $N =\PSL\bigl(2,\zz[1/2]\bigr)$.
Another computations shows that
$$
(ad)^4 \. = \. \left(\begin{array}{cc}1 & -1 \\ 0 & 1\end{array}\right),
\qquad
h \. := \, [c,[d,[b,(ad)^4]]]  \. = \. \left(\begin{array}{cc} -1 & 2 \\ 2 & - 5\end{array}\right).
$$

\begin{lemma}
The element $h$ normally generates the group~$N$.
\end{lemma}
\begin{proof}
Let $K$ be the normal subgroup of $N$ generated by~$h$. Since $N$ can be viewed as a lattice in \.
$\PSL(2,\rr) \times \PSL(2,\qqq_2)$, by Margulis's normal
subgroup theorem, we have that subgroup $K$ is either central or of finite
index, see e.g.~\cite[Ch.~IV]{Mar91}.  Moreover $N$ has congruence
subgroup property and~$h$ is not contained in any properly congruence subgroup.
\end{proof}

\smallskip

\subsection{Technical lemma}
Let $r$ denote the element
\ts $[c,[d,[b,(ad)^4]]] \in \cG$, so we have \ts $\iota(r)=h$. Let \ts $r_x = \varphi^x(r)$ \ts
be its twists by the automorphism $\varphi$ described in Definition~\ref{d:twist}.
The following lemma is a variation on our \cite[Lemma~5.16]{KP}, adjusted to
this setting.

\smallskip

\begin{lemma}
\label{l:normalsubgorups_of_F^i}
Let \. $\iota: \cG\toto \cH$, and let \. $N \lhd \cH$ \. be a normal subgroup
generated by element $r_{x_{k+1}}$, defined as above.  Then the kernel of the map \.
$F^k_\omega(\cH) \toto  \bG_{\omega,k}$ \. induced by $F^k_\omega$ from the trivial
homomorphism \. $\cH\toto\mathbf{1}$, contains \. $N^{\oplus 2^{k}}$.
Moreover, there exists a word \. $\eta_{\omega,k} \in F$ \. of length less than
$\le C\cdot 2^k$, for some universal constant $C$,
such that such the image of \.  $\eta_{\omega,k}$ in $F^k_\omega(\cH)$
normally generates \. $N^{\oplus 2^{k}}$ and \. $\eta_{\omega,k}$ \. is trivial in \.
$F^{k+1}_\om(H)$, for every $\cG\toto H$.
\end{lemma}
\begin{proof}
Consider the substitutions $\sigma,\tau$ (endomorphisms $\cG \to \cG$),
defined as follows:
\begin{itemize}
\item $\sigma(a) = aca$  and $\sigma(s) = s$, for $s\in \{b,c,d\}$\ts,
\item $\tau(a) = c$,  $\tau(b) = \tau(c) = a$ and $\tau(d) = 1$\ts.
\end{itemize}
It is easy to see that for any $\eta \in \cG$, the evaluation of $\sigma(\eta)$
in $F(\cG)$ is equal to
$$
\bigl(1;\tau(\eta),\eta\bigr) \, \in \, \{1\} \times \cG \times \cG \, \subset \, \cG \wr \zzz \ts.
$$

Define $w_i\in\cG$ for $i=0,\dots, k$ as follows:
$w_0 = r_{x_{k+1}}$ and $w_{i+1} = \sigma_{x_{k-i}}(w_{i})$
where $\sigma_{x_i} = \varphi^{x_i} \sigma \varphi^{-x_i}$
the the twist of the substitution $\sigma$.
Notice that all these words have the form $[c,[d,[b,*]]]$ because
$\sigma_{x_i}$ fixes $b$, $c$ and $d$. Therefore $\tau_x(w_i) =1$.

By construction the word $\eta_{\om,k} = w_k$ evaluates in $F^k_\om(\cG)$
to $r_{x_{k+1}}$ in one of the copies of $\cG$, viewed as a subgroup of $F^k_\om(\cG) \subset \cG \wr_{X_k} \bG_{\om,k}$.
Therefore the evaluation of $w_k$ in $F^k_\om(\cH)$ is in the kernel of $F^k_\om(\cH) \to  \bG_{\om,k}$
is the elements $h$ in one of the copies of $\cH$. This together with the transitivity of the action of $\bG_{\om,k}$ on $X_k$
shows that the kernel contains $N^{2^k}$. Finally, the word $\eta$ is trivial in  $F^k_\om(\cG)$ since the $(ad)^4 =1 $ in $F(\cG)$ for any $x$.
\end{proof}

\subsection{Main construction}

Let $\cH$ be the marked group described in~$\S$\ref{def_of_H}.
Denote $G_i$ the marked group $F^i_{(012)^\infty}(\cH)$, which surjects onto $F^i_{(012)^\infty}(\mathbf{1}) =\bG_{{(012)^\infty},i}  $
Using Corollary~\ref{cor:contracting_property_of_G} we can see that $G_i$ converge in the Chabauty topology to $\bG_\om$.

\begin{defn}
Let \. $J\subseteq 2^\nn$ \. be a fixed subset of $\nn$. Denote
$$
\wt G_J  \. := \. \bigotimes_{i\in J} G_i \quad \text{and}
\quad
 G_J \. := \. \wt G_J  \otimes \bG_\om\..
$$
By construction these are $4$-marked groups that are quotients
of~$\cG$, and s.t.\ \. $G_{\emp}=\bG_\om\ts$.
\end{defn}

Using the definition on $G_J$ one can see that the  group $G_J$ can be defined as $\bigotimes_{i\in \nn} \Gamma_{i,J}$ where
$$
\Gamma_{i,J} = \left\{ \begin{array}{ll}
G_i = F^i_{(012)^\infty}(\cH) & \mbox{if } i \in J, \\
\bG_{{(012)^\infty},i} = F^i_{(012)^\infty}(\mathbf{1}) & \mbox{if } i \not\in J.
\end{array}
 \right.
$$

\begin{lemma}
\label{lm:continuity}
The map \. $I \to G_I$ \. defines a continuous map from \. $2^{\nn}$ \.
to the space 
of marked groups. 
\end{lemma}
\begin{proof}
This follows from the convergence $G_i \to \bG_\om$ in Chabauty topology and the observation that the ball of radius $n$ in $\bigotimes \Gamma_i$ depend only on the balls of radius $n$ in $\Gamma_i$\ts, for every~$i$.  More precisely, for all~$n$, there exists $N$ such that the ball of radius $n$ in marked groups \ts $G_k$, \ts $\bG_{{(012)^\infty},k}$ \ts and \ts $\bG_{{(012)^\infty}}$ \ts coincide for all \ts $\k \geq N$. This implies that the ball of radius $n$ in $G_J$ coincide with the one in \. $\bigotimes_{i\in \nn} \wt \Gamma_{i,J}$\ts, where
$$
\wt \Gamma_{i,J} = \left\{ \begin{array}{ll}
G_i = F^i_{(012)^\infty}(\cH) & \mbox{if } i \in J, i \leq N \\
\bG_{{(012)^\infty},i} = F^i_{(012)^\infty}(\mathbf{1}) & \mbox{if } i \not\in J, i \leq  N \\
\bG_{{(012)^\infty}}  & \mbox{if }  i >  N. \\
\end{array}
\right.
$$
Therefore, the ball of radius $n$ in $G_J$ depends only on \ts $J \cap \{1,\dots, N\}$,
which implies that the function $J \to G_J$ is continuous.
\end{proof}

For any set \. $J \subset J'$ \. there is a surjection \. $G_{J'} \toto G_J$.
The main step in proving Lemma~\ref{l:main-construction} is show that
if $J$ is a proper subset of $J'$ then this map has a large kernel.

\begin{lemma}
\label{lm:splitting}
The kernel of the map $G_J \to \bG_\om$ is contains
$$
\bigoplus_{i\in J} N_i^{2^i} \subset \bigoplus_{i\in J} H_i^{2^i} \subset \bigoplus_{i\in I} G_J \,.
$$
\end{lemma}
\begin{proof}
We will use indiction on $k$ to show that
$$
\sum_{j\in J, j \leq k} N^{2^k} \subset G_J \cap \bigoplus_{i \leq k} \Gamma_{i,J}
$$
The base case of the induction $k=0$ is trivial.
Using  Lemma~\ref{l:normalsubgorups_of_F^i} we can see that element $\eta_{\omega,k}$ is trivial in $G_i$ for $i> k$, therefore it corresponds to an element in $G_J \cap \bigoplus_{i \leq k} \Gamma_{i,J}$. If $k\in J$ this elements evaluates to a normal generators in
of $N^{2^k}$ inside $G_k = \Gamma_{k,J}$ and for $i\leq k$  and $i \in J$ it to some element in $N^{2^i}$ inside $G_i = \Gamma_{i,J}$,
which finishes the induction step.
\end{proof}

\begin{proof}[Proof of Lemma~\ref{l:main-construction}]
The continuity is equivalent to Lemma~\ref{lm:continuity}. The existence of the surjection  $G_{J'} \toto G_J$ for $J \subset J'$ follows from the construction of the groups $G_J$. The strict monotonicity follows from Lemma~\ref{lm:splitting}
which gives that the kernel of $G_{J'} \toto G_J$ contains $N^{2^i}$ for any $i \in J' \setminus J$.
\end{proof}

\subsection{Proof of Lemma~\ref{l:main-mono}}
The first step is to show that for any decreasing parameter~$f$, the mapping
\ts $J \to f(G_J)$ \ts is a upper semi-continuous function \ts
$2^{\nn} \to [0,1]$.

\begin{lemma}
\label{lm:semicontunuous}
Suppose that \ts $I_n$ \ts is decreasing sequence of subsets of $\nn$ which converges to \ts
$I = \cap I_n$ \ts in the Tychonoff topology.  In other words, suppose for every \ts $n\in \nn$,
there exists \ts $k=k(n)$ \ts s.t.\ \ts $I_m \cap [n] = I \cap [n]$ \ts for all \ts $m > k(n)$.
Then for any decreasing parameter $f$ we have
$$
f(G_I) \. = \. \lim_{n\to \infty} \ts f\big(G_{I_n}\big).
$$
\end{lemma}

\begin{proof}
The mapping $I \to G_I$ defines a function from $2^{\nn}$ to the space of marked groups, which is continuous with respect to the
Tychnoff topology on $2^{\nn}$ and the Chabauty topology on the space of marked groups. Therefore,
the balls in the Cayley graphs of the groups  $G_{I_n}$ converge the  ball of $G_{I}$, which implies that
$f(G_I) \leq \lim_{n\to \infty} f(G_{I_n})$ by property~(\ref{eq:mono-limit}).

On other hand, we have that $G_I$ is a quotient of $G_{I_n}$ for each $n$.  Therefore,
\ts $f(G_I) \geq f(G_{I_n})$ by property~(\ref{eq:mono}).  Passing to the limit gives \. $f(G_I) \geq \lim_{n\to \infty} f(G_{I_n})$.
This completes the proof.
\end{proof}

One way to prove that the set of all possible values of $f$ is large is to show that $f(G_I)$
is a strictly decreasing function with respect to the lex order of $2^{\nn}$.
Unfortunately this is not the case in general, however this become true if
we restrict to sufficiently sparse subsets of~$\nn$:

\begin{lemma}
\label{l:thm-2}
Let $f$ be a strictly decreasing parameter, then
there exists a function $\mu_f: \nn \to \nn$ such that the following holds:
For every infinite set $M =\big\{ m_1,m_2, \dots, m_n,\dots \big\}$ with $m_{n+1} > \mu_f(m_n)$,
the function $J \to f(G_J)$ is a strictly decreasing function from $2^M \to \rr$
with respect to the lex total order on $2^M$.
\end{lemma}

\begin{proof}
Let $n$ be an integer and let $K$ be any subset of \. $[n-1]:= \{1,\dots, n-1\}$.
Denote \ts $K'=K \cup \{n\}$.  Since \ts $K \subsetneq K'$, by Lemma~\ref{l:main-construction}
the kernel of $G_{K'} \toto G_K$ is nonamenable and the strict monotonicity of $f$ implies that \ts
$f(G_K) > f(G_{K'})$. Let $K_m$ denote the set \. $K_m = K \cup \{m,m+1,\dots, \}$ for $m > n$.
Clearly we have that $\cap_m K_m = K$ and the sets $K_m$ converge to $K$ from above.
The semi-continuity of $f$ (Lemma~\ref{lm:semicontunuous}) implies that
$$
f(G_K) \. = \. \lim_{m \to \infty} f(G_{K_m}).
$$
Therefore, there exist \. $\Lambda_f(n,K) \in \nn$ \. such that \. $f(G_{K'}) < f(G_{K_m})$ \.
for all \. $m > \Lambda_f(n,K)$.
Denote \. $\mu_f(n) := \max_K \{\Lambda_f(n,K) \} \in \nn$,  where the maximum is taken over
all subsets of~$[n-1]$.

Let \. $M =\{ m_1,m_2, \dots, m_n,\dots \} \ssu \nn$ \. be an infinite set of integers,
such that \. $m_{n+1} > \mu_f(m_n)$ \. for all $n$.
Let \ts $J',J''$ \ts be subsets of $J$ such that \ts $J' <  J''$ \ts
in the lex order on $2^M$. By the definition of the lex order there exists $k$ such that
$J'  \cap [m_k-1] = J'' \cap [m_k-1] $ and $m_k \in J''$ but $m_k \not\in  J'$.
Using the notation from the previous paragraph with $n = m_k$ and $K=J'  \cap [m_k-1]$, we have
$K' \subset J''$ and $J' \subset K_{j_{n+1}}$.  By the choice of the function $\mu$ and the sequence $m_k$,
we have that
$$
f(G_{J'}) \geq  f(G_{K_{M(n,K)}}) > f(G_{K'}) \geq f(G_{J''}),
$$
where the first and the last inequality follow from the inclusions $J' \subset K_{M(n,K)}$ and $K' \subset J''$, and the
the second inequality follows from the definition of the constant  $\Lambda_f(n,K)$.
This verifies that $f$ is a strictly decreasing function on $2^M$, as desired.
\end{proof}

\begin{proof}[Proof of Lemma~\ref{l:main-mono}]
Let $M$ be a sparse set satisfying the conditions in the previous theorem,
so the set $2^M$ has cardinality of a continuum.   By Lemma~\ref{l:thm-2},
the function $J \to f(G_J)$ is strictly decreasing function with respect
to the total lex order on~$2^M$.  Therefore the set of all possible values
of $f$ on the groups $G_J$ has cardinality of the continuum.
\end{proof}

\medskip

\section{Isolated points} \label{s:para}

We now return to the set of spectral radii and asymptotic entropy discussed in Theorem~\ref{t:isol}.

\smallskip

\begin{lemma}\label{lm:converges}
Let $k \ge 2$, and let $\Gamma$ be a $k$-generated marked group which is not free as a marked group.
Then there exists a sequence $H_i$ of nonamenable $k$-generated marked groups
such that $H_i$ converge to an amenable marked group $H$ of subexponential growth and for any $i$ the projection
$$
\Gamma \otimes H_i \toto \Gamma
$$
has nonamenable kernel.
\end{lemma}

\begin{rem}
The condition that $\Gamma$ is not free is necessary since if \. $\Gamma = \F_k$\ts,
then for every $k$-generated group $H$ we have the map \.
$\F_k \otimes H \toto \F_k$ \. is an isomorphism.
\end{rem}

\begin{rem}
For $k=4$, if we drop the last condition in the lemma, one can simply take the group \.
$G_i=F^i_{(012)^\infty}(\cH)$ \. constructed in Section~\ref{s:construction}.
However, since all these groups are quotients of~$\cG$, we note that the
last condition fails for \ts $\Gamma = \cG$.
\end{rem}

As an immediate corollary of this lemma we obtain the following result.

\begin{thm}
\label{t:isol-generic}
Let $f$ be a sharply decreasing parameter.  Then, for any marked group $(\Gamma,S)$
which is not free, there exists a sequence of marked groups \. $\{(\Gamma_i,S_i)\}$,
s.t. $f(\Gamma_i,S_i)< f(\Gamma,S)$ \. and \. $\lim_{i\to \infty} f(\Gamma_i,S_i)= f(\Gamma,S)$.
Therefore, the set \ts $X_{f,2k}$ \ts of values of $f$ on all marked groups with
a generating set of size $2k$ has no isolated points, except possibly at \ts $f(\F_k)$.
\end{thm}

\begin{proof}
This is an immediate consequence of Lemma~\ref{lm:converges}, with $\Gamma_i = \Gamma \otimes H_i$ --
the inequality $f(\Gamma_i,S_i)< f(\Gamma,S)$ follow from the strict property of $f$ and the fact that the kernel of
$\Gamma \otimes H_i \toto \Gamma$ is nonamenable. The property $H_i \to H$ in the Chabauty topology implies that \.
$\Gamma \otimes H_i \to \Gamma \otimes H$ \. in the  in the Chabauty topology.  Therefore, \.
$\limsup_{i\to \infty} f(\Gamma \otimes H_i)\geq  f(\Gamma \otimes H)$.
However the sharpness of $f$ implies that $f(\Gamma \otimes H) = f(\Gamma)$ since
the kernel of $\Gamma \otimes H \toto H$ is a subgroup of $H$, and therefore it is amenable.
\end{proof}

\begin{proof}[Proof of Lemma~\ref{lm:converges}]
We start with the following group theoretic result.

\begin{clm}
For any nontrivial normal subgroup $\Delta \lhd \F_k$, there exist a finite $k$-generated marked group $K$ such that the kernel $R= \ker (\F_k \toto K)$ admits a surjective homomorphism $\psi: R \toto  \cG$ where $\psi(\Delta \cap R)$ contains the generator $a$ of $\cG$.
\end{clm}

\begin{proof}[Proof of the Claim]
Since $\Delta$  is nontrivial there exists a nontrivial $r \in \Delta \cap [\F_k,\F_k]$.
Consider the $2$-Frattini series for $\F_k$ defined by $T_0 = \F_k$ and $T_i = T_i^2[T_i,T_i]$. Each $T_i$ is a finite index in $\F_k$ and
$\bigcap T_i = 1$ since the free group is a residually $2$-group.  Therefore, there exists an index $i$ such that $r \in T_i \setminus T_{i+1}$ and by the choice of $r$ we have that $i \geq 1$. We can take $K = \F_k/T_i$, the kernel $R$ is $T_i$  and is a free group
of rank $(k-1) |K| + 1 > 3$. Since $r$ is outside the $2$-Frattini subgroup of $R$, there exists a homomorphism $R \toto \zz_2 * \zz_2 * \zz_2$ which maps $r$ to one of the generators. Finally $\psi$ is the composition with the projection  $\zz_2 * \zz_2 * \zz_2 \toto \cG$.
\end{proof}

Let $\psi$ be the homomorphism from the Claim applied to $\Delta = \ker (\F_k \toto \Gamma)$.
Denote by $\psi_i$ the composition of $\psi$ and the projection $\cG \toto F^i_{(012)^\infty}(\cH)$ and $\psi_\infty$ the composition of $\psi$ and the projection $\cG \toto \bG$.
Let $L_i$ be in intersection of all conjugates of $\ker \psi_i$ in $F_k$ (there are exactly $|K|$ of them since $\ker \phi_i$ is a normal subgroup of $R$, similarly define $L_\infty$.
By construction we have that $H_i = F_k/L_i$ fits into exact sequence
$$
1 \to M_i \to H_i \to K \to 1
$$
where $M_i$ is naturally a subgroup of $F^i_{(012)^\infty}(\cH)^{\oplus |K|}$, and a similar statement for $H_\infty$.

By construction there is a relation in $\Gamma$ which maps to a generator of $\cG$ therefore there exists a relation in $\Gamma$ which maps to the word $\eta_{(012)^\infty,i}$ constructed in the proof of Lemma~\ref{l:normalsubgorups_of_F^i}. This implies that
the kernel of the map
$$
\Gamma \otimes H_i \toto \Gamma
$$
contains $N^{2^i}$ and is nonamenable.

The convergence \ts $\cG \toto F^i_{(012)^\infty}(\cH) \to \bG$ \ts
in the Chabauty topology implies that $H_i$ converge to $H_\infty$
which is an amenable group.  This completes the proof of Lemma~\ref{lm:converges}.
\end{proof}

\smallskip

\begin{proof}[Proof of Theorem~\ref{t:isol}]
For $k$ even and \ts $x> \al_k$\ts,
the theorem follows immediately from Theorem~\ref{t:isol-generic} since
the spectral radius is sharply decreasing by Lemma~\ref{l:radius-strict}.
The case for $k$ odd follows from a variant of Lemma~\ref{lm:converges} for groups where one of the generators has order $2$.

For \ts $k=2r\ge 4$, it is known that $\al_k$ is not an isolated point because
$$
\al_k \, = \, \lim_{n \to \infty} \. \rho\big(G_{r,n}\ts,S\big),
$$
where \ts $G_{r,n}$ \ts is the free product of \ts $r$ \ts copies of \ts $\zz_n$ \ts
with a natural generating set~$S$,  see e.g.\ \cite[Ex.~9.25(2)]{Woe}.
The argument in the reference above can be extended to show that for all \ts $k=2r+1\ge 3$, we have:
$$
\al_k \, = \, \lim_{n \to \infty} \. \rho\big(Y_{r,n},S\big)\., \quad \text{where} \quad
Y_{r,n} \, := \ \underbrace{\zz_2 \ast \. \cdots \. \ast \zz_2}_{2r} \. \ast \.\ts \zz_n \..
$$
This shows that \ts $\al_k$ \ts are not isolated points, for all \ts $k\ge 3$. This completes the proof.
\end{proof}

\medskip

\section{Other monotone parameters}\label{s:other}

\subsection{Cheeger constant} \label{ss:other-isop}
Let $|S|=k$ and define the \emph{Cheeger constant} (also called \emph{isoperimetric constant}):
$$
\phi(G,S) \, := \, \inf_{X\ssu G} \. \frac{|E(X,G\sm X)|}{k\ts |X|}\.,
$$
where the infimum is over all nonempty finite $X$, and \ts $E(X,Y)$ \ts
is the set of edges $(x,y)$ in \ts $\text{Cay}(G,S)$ \ts
such that $x\in X, \ts y\in Y$.  Note that \. $0\le \phi(G,S) \le 1$.

The following celebrated inequality relates the spectral radius
and the Cheeger constant:
$$
\tfrac12 \.\phi(G,S)^2 \, \le \, 1 \. - \. \rho(G,S) \, \le \, \phi(G,S).
$$
This inequality was discovered independently by a number of authors in different
contexts, see e.g.~\cite[$\S$7.2]{Pete}.  In particular, we have \ts $\phi(G,S)>0$ \ts
if and only if $G$ is nonamenable.

Similar to other probabilistic parameters, computing the Cheeger constant
is very difficult and the exact values are known only in a few special cases.
For example,  $\phi(\F_k,S) = \frac{k-1}{k}$  \. for a free group with standard
generators.  Additionally, the Cheeger constant is computed for nonamenable
hyperbolic tessellations where it is always algebraic~\cite{HJL02}
(see also~\cite[$\S$5.3]{LP16}).

\begin{prop}\label{p:isop-mono}
Cheeger constant \. $\phi: \{(G,S)\} \to \rr_{\ge 0}$ \. is an increasing parameter.
\end{prop}

The proof of this result is straightforward.  By analogy with Lemma~\ref{l:radius-strict},
one can ask if $\phi$ is strictly increasing?  Perhaps even sharply increasing?
Unfortunately, this remains open:

\begin{conj}
\label{conj:isop-strict}
Let \ts $G=\<S\>$, let $N$ be a finitely generated normal subgroup of~$G$,
and let $S'$ be the projection of $S$ onto~$G/N$. Then \ts $\phi(G,S) < \phi(G/N,S')$
\. \underline{if and only if} \. $N$ is nonamenable.
\end{conj}

By analogy with the proof of our Theorem~\ref{t:main}, the \emph{if} part
of the conjecture gives:

\begin{prop}  \label{p:isop-main}
Conjecture~\ref{conj:isop-strict} implies that
there is a family of Cayley graphs
\. $\big\{(G_\om,S_\om)\big\}$, such that the set of values \.
$\big\{\phi(G_\om,S_\om)\big\}$ \. has cardinality
of the continuum.
\end{prop}

The conclusion of the proposition remains an
open problem in the area.  The analogue of Theorem~\ref{t:isol}
for \ts $X_{\phi,k}$ \ts follow from the conjecture.
We omit the details.

\smallskip

\subsection{Speed of simple random walks} \label{ss:other-speed}
As in the introduction, let \ts $\{x_n\}$ \ts denote simple random walk
on \ts $\Ga = \text{Cay}(G,S)$ \ts starting at \ts $x_0=e$.  For all $z\in G$,
denote by \ts $|z|$ \ts the distance from~$e$ to~$z$ in~$\Ga$.
The \emph{speed} (also called \emph{drift} and \emph{rate of escape})
of \ts $\{x_n\}$ \ts is defined as follows:
$$
\si(G,S) \, := \, \lim_{n\to \infty} \. \frac{\mathbb{E} \ts [\ts |x_n| \ts ]}{n}\.,
$$
see e.g.~\cite[$\S$9.1]{Pete} and \cite[$\S$II.8]{Woe}.  Note that
the limit is well defined and the same for almost all sample paths.
Famously, it was shows in \cite{KL07} \ts that \ts $\si(G,S)>0$ \ts if and only if \ts
$h(G,S)>0$.   We also have:

\begin{prop}\label{p:speed-mono}
Speed \. $\si: \{(G,S)\} \to \rr_{\ge 0}$ \. is an increasing parameter.
\end{prop}

The proof of this result is straightforward.  Unfortunately, the following
natural analogue of Lemma~\ref{l:entropy-strict} remains open:

\begin{conj}
\label{conj:speed-strict}
Let \ts $G=\<S\>$ \ts and let $N$ be a finitely generated
normal subgroup of~$G$.  Suppose $N$ is nonamenable,
and let $S'$ be the projection of $S$ onto~$G/N$.
Then \ts $\si(G,S) < \si(G/N,S')$.
\end{conj}

By analogy with the proof of our Theorem~\ref{t:main} and Proposition~\ref{p:isop-main} above,
we have:

\begin{prop}  \label{p:speed-main}
Conjecture~\ref{conj:speed-strict} implies that
there is a family of Cayley graphs
\. $\big\{(G_\om,S_\om)\big\}$, such that the set of values \.
$\big\{\si(G_\om,S_\om)\big\}$ \. has cardinality
of the continuum.
\end{prop}

\medskip

\subsection{Rate of exponential growth} \label{ss:other-exp}
Let \ts $b(n) := |\{z\in G \. : \. |z|\le n\}|$ \ts denote the number of elements
at distance at most~$n$, and define the \emph{rate of exponential growth} as follows:
$$
\ga(G,S) \, := \, \frac{1}{|S|} \. \lim_{n\to \infty} \. \frac{\log b(n)}{n}\..
$$
Clearly, \ts $\ga(G,S)\in [0,1)$, and \ts $\ga(G,S)>0$ \ts is independent of
the generating set~$S$.  See e.g.\ \cite{Gri-one,GP08,dlH00} for accessible introductions.

\begin{prop}\label{p:exp-mono}
Rate of exponential growth \. $\ga: \{(G,S)\} \to \rr_{\ge 0}$ \. is an increasing parameter.
\end{prop}

The proof of this result is straightforward.  Clearly, the rate of exponential
growth is not strictly increasing.  In fact, the notion of strictly increasing
parameters for the rate of exponential growth is closely related to the notion
of \emph{growth tightness} introduced
in \cite[$\S$2]{GH} and established in \cite{AL02} for word hyperbolic groups.
In \cite{Ers04}, Erschler was able to modify the notion of ``strict monotonicity''
to obtain the following natural analogue of Theorem~\ref{t:main} for the rate
of exponential growth:

\begin{thm}[\cite{Ers04}]  \label{t:exp-main}
The set of rates of exponential growth  \.
$\big\{\ga(G,S)\big\}$ \. has cardinality
of the continuum.
\end{thm}

It is worth comparing this question with Grigorchuk's
celebrated result in \cite{Gri3}, that there is an uncountable family of
marked groups with incomparable growth functions.  The proof of
Theorem~\ref{t:exp-main} is also somewhat related to Bowditch's
elementary construction in \cite{Bow98}, of an uncountable family of
marked groups with pairwise non-quasi-isometric Cayley graphs.
See also the most recent result by Louvaris, Wise and Yehuda \cite{LWY24},
which proves that the set of (unscaled) growth rates of subgroups of the
free group $\F_k$ is dense on $[1,2k-1]$.

\medskip

\subsection{Connective constant} \label{ss:other-SAW}
Let $c(n)$ denote the number of \emph{self-avoiding walks} of length~$n$
in the Cayley graph \ts $\text{Cay}(G,S)$, and define the
\emph{connective constant} \ts as follows:
$$
\mu(G,S) \, := \, \lim_{n\to \infty} \. \sqrt[n]{\upsilon(n)} \..
$$
Here the limit exists by submultiplicativity: \ts $\upsilon(m+n) \le \upsilon(m) \ts \upsilon(n)$,
see e.g.\ \cite[$\S$1.2]{MS93}.  The exact value is known for the hexagonal
lattice \cite{DS12} and only few other special cases, see e.g.\ \cite[$\S$1.2]{GL19}.

\begin{thm}[{\cite[Cor.~4.1]{GL14}}{}]  \label{t:SAW-GL}
Connective constant \. $\mu:  \{(G,S)\} \to \rr_{\ge 1}$ \. is a strictly
increasing parameter.  Moreover, the inequality~\eqref{eq:mono} is strict
for all \ts $N\ne \mathbf{1}$.
\end{thm}

Now Lemma~\ref{l:main-mono} gives a new proof of the following known result:

\begin{thm}[{\cite{Mar17}}{}]  \label{t:SAW-Mar}
There is a family of \ts $4$-regular Cayley graphs
\. $\big\{\ts \text{\em Cay}(G_\om,S_\om)\ts\big\}$, such that the set
of connective constants \.
$\big\{\mu(G_\om,S_\om)\big\}$ \. has cardinality of the continuum.
\end{thm}

The original proof by Martineau \cite{Mar17} is similar in nature
to our approach but using weaker tools, namely \cite[Lemma~III.40]{dlH00}
which suffices for the case of connective constants (cf.~$\S$\ref{ss:finrem-par}).

\medskip

\section{Final remarks and open problems}\label{s:finrem}

\subsection{Historical notes} \label{ss:finrem-hist}
The results of this paper (for the spectral radius) were announced over seven
years ago.\footnote{Martin~Kassabov, \emph{A nice trick involving amenable groups},
MSRI talk (Dec.~9, 2016), video and slides available at \ts
\href{https://www.slmath.org/workshops/770/schedules/21638}{www.slmath.org/workshops/770/schedules/21638}}
We are happy that we could extend them to other monotone parameters, and
create an alternative to the small cancellation theory approach (see below).
In the meantime, our construction influenced \cite{TZ19} mentioned
in the introduction.
While this creates a messy timeline, we would like to
acknowledge that the asymptotic entropy version of Theorem~\ref{t:main}
can be attributed to Tamuz and Zheng at least as much as to this work.

Let us mention that the idea of taking products of marked groups
goes back to B.H.~Neumann \cite{Neu37}, and was repeatedly used in the
last two decades.  Besides \cite{KP}, let us single out
Pyber's \cite{Pyb04} (see also \cite{Pyb03})
with an uncountable family of groups related to
the \emph{Grothendieck problem}, Erschler's \cite{Ers06} with a
closely related construction based on decorated Grigorchuk groups,
and the most recent breakthrough \cite{BZ21} by Brieussel and Zheng.

\subsection{Graph theoretic aspects} \label{ss:finrem-graph}
In~\cite{LM06}, Leader and Markstr\"om constructed a simple uncountable family
of pairwise nonisomorphic $4$-regular Cayley graphs.  They were clearly
unaware of the earlier works by Grigorchuk~\cite{Gri3}, Bowditch~\cite{Bow98}
and Erschler~\cite{Ers04} which prove stronger results.  However,
the elementary nature of their construction is of independent interest.

\subsection{Computability aspects} \label{ss:finrem-comp}
In~\cite{Hag08}, H\"aggstr\"om showed that the critical probability \ts
$p^\site_c(\zz^2)$ \ts is \emph{computable} \ts in the sense of the
\emph{Church--Turing thesis}: there exists a Turing machine which computes
the digits in the binary expansion.  This resolved Toom's question.
H\"aggstr\"om then asked if \ts
$p_c(G,S)$ \ts is always computable (ibid., p.~323).  The negative
answer follows immediately from our Theorem~\ref{t:main}
and the observation that there are countably many Turing machines.

\subsection{Set theoretic aspects} \label{ss:finrem-sets}
As in the introduction, let \. $X_{\rho,k} \subset  (0,1]$ \. denote the
set of spectral radii of marked groups with $k$ generators, and let
\. $X_\rho :=\cup X_{\rho,k}$\ts.

\begin{OP}\label{op:rho}
We have: \. $X_\rho \. = \. (0,1]$.
\end{OP}

Between ourselves, we disagree whether one should believe or disbelieve this claim.
While our results seem to suggest a positive answer, they give no intuition
whether $X_\rho$ is closed or dense.

Now, Main~Theorem~\ref{t:main} shows that \ts $X_{\rho,8}$ \ts has cardinality
of the continuum.  Our proof implies a stronger result, that \ts $X_{\rho,8}$ \ts
has an embedding of the \emph{Cantor set} (see Lemma~\ref{l:main-construction}).
On the surface, this appears a stronger claim since there is a natural construction
of the \emph{Bernstein set} which has cardinality of the continuum and no embedding
of the Cantor set, see e.g.\ \cite[Ex.~8.24]{Kechris}.

Looking closer at our results, for even $k$, we prove in Theorem~\ref{t:isol}
that \ts $X_{\rho,k}$ \ts has no isolated points in the interval
\ts $[\al_k,1]$.  If \ts $X_{\rho,k}$ \ts is closed, then it is a perfect
Polish space that always contains the Cantor set, see e.g.\ \cite[Thm~6.2]{Kechris}.
In our case, it is easy to see that \ts $X_{\rho,k}$ \ts is a projection
of a Borel set, and thus \emph{analytic}, see e.g.\ \cite[$\S$14.A]{Kechris}.
Recall that every analytic set that is a subset of a Polish space either is countable,
or contains a Cantor set, see e.g.\ \cite[Ex.~14.13]{Kechris}.

In other words, very general set theoretic arguments imply that
set \ts $X_{\rho,k}$ \ts is either countable or contains a Cantor set,
and thus has the cardinality of the continuum.  This argument applies to
other monotone parameters in Theorem~\ref{t:main}, suggesting that containment of the Cantor set
is unsurprising.

\subsection{Explicit constructions} \label{ss:finrem-explicit}
The proof of Lemma~\ref{l:main-construction} is fundamentally set theoretic
and does not allow an \emph{explicit construction} \ts of the Cayley graph with a
transcendental spectral radius: \ts $\rho(G,S) \notin \ov \qqq$.  Here we
are intentionally vague about the notion of ``explicit construction'',
as opposed to the setting in graph theory where it is well defined,
see e.g.\ \cite[$\S$2.1]{HLW06}
and \cite[$\S$9.2]{Wig19}.  This leads to a host of interrelated open problems
corresponding to different possible meanings in our context.

\begin{question}\label{q:growth-transc}
Is there a finitely presented group with a transcendental spectral radius?
What about recursively presented groups?  Similarly, what about graph
automata groups?~\footnote{Soon after this paper was posted on the \ts {\tt arXiv},
a positive answer to the first two questions was given in~\cite{Bod24}.}
\end{question}

Despite our efforts, we are unable to resolve either of these questions
using our tools.
Note that there is a closely related but weaker notion of \emph{D-transcendental
cogrowth series}, see e.g.\ \cite{GH,GP17}. We refer to \cite{Pak18}
for some context about problems of counting walks in graphs
and further references.

\subsection{Asymptotic versions} \label{ss:finrem-inter}
When the group is amenable, one can ask about the asymptotics
of the return probability and the (closely related) \emph{isoperimetric profile}.
Similarly, when the Cayley graph has no non-constant bounded harmonic
functions, one can ask about the asymptotics of the speed and entropy
functions.  In these setting, there is a large literature on both the
exact and oscillating growth of these functions, too large to be reviewed
here.  We refer to \cite{BZ21} for the recent breakthrough and many
references therein.

For the critical probability, one can ask about asymptotics of the
number of cuts, see e.g.\ \cite{Tim07}.  We note that this asymptotic
version is always exponential for vertex-transitive graphs, and is
largely of interest for families of graphs with nearly linear growth.

\subsection{Critical probability on general graphs} \label{ss:finrem-perc}
The problem of describing critical probability constants on general
graphs was introduced by van den Berg \cite{vdB82}, who showed that
every \ts $p \in [0,1]$ \ts is a critical probability using a
probabilistic method.  Famously, Grimmett \cite{Gri83} showed
that natural subgraphs of \ts $\zz^2$ \ts can be used to obtain every
\ts $p \in [\frac12,1]$ \ts as a critical probability of
the bond percolation.   In a different direction, Ord, Whittington
and Wilker \cite{OWW84} construct a countable family of graphs using
decorations of~$\zz^2$, which has \ts $p_c^\bond$ \ts dense
on~$(0,1)$.  While neither of these constructions is vertex-transitive,
they suggest the following

\begin{question}  \label{q:perc-bond}
Let \ts $Y:= \cup_k \ts X_{p_c^\bond,k}$ \ts denote the set of critical probabilities
of all Cayley graphs of bounded degree.  Does there exist a constant \ts
$\al>0$ \ts such that \ts  $Y\cap (0,\al)$ \ts is dense?  Does there exist
a constant \ts $\be >0$ \ts such that \ts  $(0,\be) \ssu Y\ts$?
\end{question}

\subsection{Monotone properties} \label{ss:finrem-par}
The notion of \emph{monotone properties} are modeled after standard
notions of monotone and hereditary properties in probabilistic
combinatorics, which describe set systems closed under taking
subsets.  Typical examples include properties of graphs that are
invariant under deletion of edges or vertices, see e.g.\
\cite[$\S$6.3,~$\S$17.4]{AS16}.
Similarly, \emph{parameters} are standard in graph theory, and
describe any of the numerous quantitative graph functions,
see e.g.\ \cite{Bon95,Lov12}.

Finally, note that both critical probabilities and connective
constants are \emph{completely monotone}, i.e.\ the
inequality~\eqref{eq:mono} is strict for \emph{all} nontrivial~$N$
(see Lemma~\ref{l:perc-strict} and Theorem~\ref{t:SAW-GL}).
For the completely monotone parameters, the proof of Theorem~\ref{l:main-mono}
substantially simplifies (cf.~\cite{Mar17}), while Theorem~\ref{t:isol-generic}
is no longer valid.  As we mentioned in the introduction, it is false
for critical probabilities.

\subsection{Small cancellation groups} \label{ss:finrem-small}
There is an alternative approach to monotone parameters coming
from the small cancellation theory.  Notable highlights include
Bowditch's work \cite{Bow98} mentioned in~$\S$\ref{ss:other-exp},
and Erschler's Theorem~\ref{t:exp-main}.

After the results of this paper were obtained, Erschler showed
us how to prove both Lemma~\ref{l:main-construction} and
Theorem~\ref{t:main} using the construction from \cite{Ers04}
combined with strictly monotone properties.\footnote{Anna Erschler,
personal communication.}
Moreover, Osin suggested that this approach could also
be used to have groups satisfy additional properties, such as acylindrically
hyperbolic, lacunary hyperbolic, and  property~(T).\footnote{Denis Osin,
personal communication.}
It would be interesting to see how much further this
construction can be developed.
We note, however, that our approach is nearly
self-contained (modulo some lemmas proved in \cite{KP}).
Note also that Theorem~\ref{t:isol} seems not attainable
via small cancellation groups.

\vskip.6cm

\subsection*{Acknowledgements}
We are very grateful to Jason Bell, Corentin Bodart,
Nikita Gladkov, Slava Grigorchuk, Tom Hutchcroft,
Vadim Kaimanovich, Anders Karlsson, Andrew Marks,
Justin Moore, Tatiana Nagnibeda, G\'abor Pete,
Omer Tamuz, Romain Tessera
and Tianyi Zheng, for many interesting remarks and help with
the references.  Special thanks to Anna Erschler and Denis Osin
for telling us about the small cancelation group approach.
Both authors were partially supported by the NSF.


\end{document}